\documentclass[11pt]{amsart}

\usepackage{enumerate,kbordermatrix}

\usepackage[colorlinks,linkcolor=blue,anchorcolor=blue,citecolor=blue,filecolor=blue,menucolor=blue,runcolor=blue,urlcolor=blue]{hyperref}
\usepackage{graphicx}

\usepackage[bitstream-charter]{mathdesign}
\usepackage{bm}% bold symbols
\newtheorem{thm}{Theorem}[section]
\newtheorem{lemma}[thm]{Lemma}
\newtheorem{theorem}[thm]{Theorem}
\newtheorem{corollary}[thm]{Corollary}
\newtheorem{prop}[thm]{Proposition}
\newtheorem{claim}{Claim}[thm]

\theoremstyle{definition}
\newtheorem{definition}[thm]{Definition}

\DeclareMathOperator{\GF}{GF}
\DeclareMathOperator{\ex}{Ex}
\DeclareMathOperator{\AG}{AG}

\newcommand{\del}{\hspace{-0.75pt}\backslash}
\newcommand{\con}{\hspace{-0.75pt}/}
\newcommand{\nfd}{\ensuremath{(F_{7}^{-})^{*}}}
\newcommand{\agde}{\ensuremath{\AG(2, 3)\del e}}
\newcommand{\qalpha}{\ensuremath{\mathbb{Q}(\alpha)}}
\newcommand{\Z}{ \mathbb{Z} }
\newcommand{\field}{\ensuremath{\mathbb{F}}}
\newcommand{\psru}{\mathbb{S}}
\newcommand{\nreg}{\mathbb{U}_1}
\newcommand{\rank}{\ensuremath{r}}
\newcommand{\corank}{\ensuremath{\rank^{*}}}
\DeclareMathOperator{\closure}{cl}
\DeclareMathOperator{\fullclosure}{fcl}
\newcommand{\coclosure}{\ensuremath{\closure}^*}

\newcommand{\wheel}[1]{M(\mathcal{W}_{#1})}
\newcommand{\whirl}[1]{\mathcal{W}^{#1}}

\newcommand{\growfan}[3]{\ensuremath{{\boxtimes}_{#1}^{#2}(#3)}}

\newcommand{\smin}{-}       % Because journals appear to change \setminus back to the old def.

\begin{document}

\title[On the relative importance of excluded minors]
{On the relative importance of excluded minors}

\author[Hall]{Rhiannon Hall}
\address{School of Information Systems, Computing and Mathematics,
Brunel University, Uxbridge UB8 3PH, United Kingdom}
\email{rhiannon.hall@brunel.ac.uk}

\author[Mayhew]{Dillon Mayhew}
\address{School of Mathematics, Statistics, and Operations Research\\
Victoria University of Wellington\\
New Zealand}
\email{dillon.mayhew@msor.vuw.ac.nz}
\thanks{The second author was supported by a Foundation for
Research Science \& Technology post-doctoral fellowship.}

\author[Van Zwam]{Stefan H. M. van Zwam}
\address{Department of Mathematics\\ Princeton University\\ United States}
\email{svanzwam@math.princeton.edu}
\thanks{The third author was supported by NWO}

\date{\today}

\begin{abstract}
  If $\mathcal{E}$ is a set of matroids, then $\ex(\mathcal{E})$ denotes the set of matroids that have no minor isomorphic to a member of $\mathcal{E}$. If $\mathcal{E}'\subseteq \mathcal{E}$, we say that
  $\mathcal{E}'$ is \emph{superfluous} if
  $\ex(\mathcal{E}-\mathcal{E}')-\ex(\mathcal{E})$ contains
  only finitely many $3$-connected matroids.
  We characterize the superfluous subsets of six well-known
  collections of excluded minors.
\end{abstract}

\maketitle

%%%%%%%%%%%%%%%%%%%%%%%%%%%%%%%%%%%%%%%%%%%%%%%%%%%%%%%%%%%%%%%%%%%%
%%%%%%%%%%%%%%%%%%%%%%%%%%%%%%%%%%%%%%%%%%%%%%%%%%%%%%%%%%%%%%%%%%%%
\section{Introduction}\label{sct1}
%%%%%%%%%%%%%%%%%%%%%%%%%%%%%%%%%%%%%%%%%%%%%%%%%%%%%%%%%%%%%%%%%%%%
%%%%%%%%%%%%%%%%%%%%%%%%%%%%%%%%%%%%%%%%%%%%%%%%%%%%%%%%%%%%%%%%%%%%

If $\mathcal{E}$ is a set of matroids, then let $\ex(\mathcal{E})$
be the set of matroids such that $M\in \ex(\mathcal{E})$ if and only if $M$ has no minor isomorphic to a member of $\mathcal{E}$.
Thus, if
$\mathcal{P}=\{U_{2,4},F_{7}, F_{7}^{*},M(K_{3,3}), M(K_{5}), M^{*}(K_{3,3}), M^{*}(K_{5})\}$, then 
$\ex(\mathcal{P})$ is the set of graphic matroids of planar graphs.
Hall's classical theorem on the graphs without a
$K_{3,3}$-minor \cite{Hal43} can be interpreted as saying
that
\[\ex(\mathcal{P}-\{M(K_{5})\})-\ex(\mathcal{P})\]
contains only a single $3$-connected matroid, namely
$M(K_{5})$ itself.
This motivates the following definition:
if $\mathcal{E}$ is a set of matroids, then
$\mathcal{E}'\subseteq \mathcal{E}$ is a \emph{superfluous}
subset of $\mathcal{E}$ if
$\ex(\mathcal{E}-\mathcal{E}')-\ex(\mathcal{E})$ contains
only finitely many $3$-connected matroids.
Thus $\{M(K_{5})\}$ is a superfluous subset of
$\mathcal{P}$.
Obviously every subset of a superfluous subset is itself
superfluous.
In this article we characterize the superfluous subsets
of six well-known collections of excluded minors.

We will concentrate on the excluded minors for classes of matroids
representable over partial fields.
Partial fields were introduced by Semple and Whittle \cite{SW96},
prompted by Whittle's investigation of classes of
ternary matroids \cite{Whi95,Whi97}.
A \emph{partial field} is a pair $(R,G)$, where $R$ is a commutative
ring with identity, and $G$ is a subgroup of the multiplicative
group of $R$, such that $-1\in G$.
Note that every field, $\field$, can be seen as a partial field,
$(\field, \field-\{0\})$.
For more information on partial fields, and matroid representations
over them, we refer to \cite{PZ08lift}.

To date, the class of matroids representable over a partial
field has been characterized via excluded minors in only six
cases.
Those cases are: the fields $\GF(2)$, $\GF(3)$, and $\GF(4)$,
the regular partial field, and two of the partial fields
discovered by Whittle, namely the sixth-roots-of-unity
partial field, and the near-regular partial field.
We will characterize the superfluous subsets of all these
collections of excluded minors.

First of all, Tutte \cite{Tuthom} showed that the only excluded minor
for the class of $\GF(2)$-representable matroids is
$U_{2,4}$.
It is clear that the only superfluous subset in this case is
the empty set.
For a more interesting example, we examine the
\emph{regular} partial field,
$\mathbb{U}_{0}:=(\mathbb{Z},\{1,-1\})$.
Tutte also proved that the set of excluded minors for
$\mathbb{U}_{0}$-representable matroids is
$\{U_{2,4},F_{7},F_{7}^{*}\}$.
It is a well-known application of Seymour's Splitter Theorem
\cite{Sey80} that $F_{7}$ is a splitter for
the class $\ex(\{U_{2,4},F_{7}^{*}\})$.
The next theorem follows easily from this fact.

\begin{theorem}
\label{thm:regsuper}
The only non-empty superfluous subsets of $\{U_{2,4},F_{7},F_{7}^{*}\}$
are $\{F_{7}\}$ and $\{F_{7}^{*}\}$.
The only $3$-connected matroid in
$\ex(\{U_{2,4},F_{7}^{*}\})-\ex(\{U_{2,4},F_{7},F_{7}^{*}\})$
is $F_{7}$.
\end{theorem}

Next we consider the excluded-minor characterization of
$\GF(3)$-representable matroids, due to Bixby and Seymour
\cite{Bix79, Sey79}.

\begin{theorem}
\label{thm:gf3excluded}
The set of excluded minors for $\GF(3)$-representable matroids
is $\{U_{2,5},U_{3,5},F_{7},F_{7}^{*}\}$.
\end{theorem}

\begin{theorem}
\label{thm:gf3super}
The only non-empty superfluous subsets of
$\{U_{2,5},U_{3,5},F_{7},F_{7}^{*}\}$ are
$\{F_{7}\}$ and $\{F_{7}^{*}\}$.
The only $3$-connected matroid in
$\ex(\{U_{2,5},U_{3,5},F_{7}^{*}\})-\ex(\{U_{2,5},U_{3,5},F_{7},F_{7}^{*}\})$
is $F_{7}$.
\end{theorem}

The set of excluded minors for $\GF(4)$-representable
matroids was characterized by Geelen, Gerards, and Kapoor 
\cite{GGK}.

\begin{theorem}\label{thm:gf4excluded}
The set of excluded minors for the class of $\GF(4)$-representable
matroids is $\{U_{2,6},U_{4,6},F_7^-,\nfd,P_6,P_8,P_8^{=}\}$.
\end{theorem}

Let $\mathcal{O}$ be the set of excluded minors in
Theorem \ref{thm:gf4excluded}.
Geelen, Oxley, Vertigan, and Whittle showed the following:

\begin{theorem}[{\cite[Theorem 1.1]{GOVW00}}]\label{thm:GOVW02}
  Let $M$ be a $3$-connected matroid. Then one of the following holds:
  \begin{enumerate}
    \item $M$ is $\GF(4)$-representable;
    \item $M$ has a minor isomorphic to one of $\mathcal{O}\smin \{P_8,P_8^{=}\}$;
    \item $M$ is isomorphic to $P_8^{=}$;
    \item $M$ is isomorphic to a minor of $S(5,6,12)$.
  \end{enumerate}
\end{theorem}

This implies that $\{P_{8},P_{8}^{=}\}$ is a superfluous
subset of $\mathcal{O}$.
We complement this theorem by showing that it is best possible:

\begin{theorem}\label{thm:GOVWcomplement}
The only superfluous subsets of $\mathcal{O}$ are the subsets
of $\{P_{8},P_{8}^{=}\}$.
The only $3$-connected matroids in
$\ex(\mathcal{O}-\{P_{8},P_{8}^{=}\})-\ex(\mathcal{O})$
are isomorphic to $P_{8}^{=}$, or minors of
$S(5,6,12)$.
\end{theorem}

Let $\psru:=(\mathbb{C},\{z\in\mathbb{C}\mid z^{6}=1\})$ be the
\emph{sixth-roots-of-unity} partial field, so that
a matroid is $\psru$-representable if and only if it
is both $\GF(3)$- and $\GF(4)$-representable.
By combining Theorems \ref{thm:gf3excluded} and
\ref{thm:gf4excluded}, 
Geelen, Gerards, and Kapoor derived the following result
\cite[Corollary 1.4]{GGK}.

\begin{theorem}
\label{thm:sruexcluded}
The set of excluded minors for the class of
$\psru$-representable matroids is
$\{U_{2,5},U_{3,5},F_{7},F_{7}^{*},F_{7}^{-},\nfd,P_{8}\}$.
\end{theorem}

Let $\mathcal{S}$ be the set of excluded minors in
Theorem \ref{thm:sruexcluded}.

\begin{theorem}
\label{thm:srusuper}
The only superfluous subsets of $\mathcal{S}$ are the
subsets of $\{F_{7},P_{8}\}$ and $\{F_{7}^{*},P_{8}\}$.
The only $3$-connected matroids in
$\ex(\mathcal{S}-\{F_{7},P_{8}\})-\ex(\mathcal{S})$
are isomorphic to $F_{7}$, or minors of
$S(5,6,12)$.
\end{theorem}

Let $\nreg:=(\qalpha,\{\pm\alpha^{i}(1-\alpha)^{j}\mid i,j\in \Z\})$
be the \emph{near-regular} partial field.
A matroid is $\nreg$-representable if and only if it
is representable over $\GF(3)$, $\GF(4)$, and $\GF(5)$.
The next theorem is proved in \cite{HMZ11}.

\begin{theorem}
\label{thm:nregexcluded}
The set of excluded minors for the class of
$\nreg$-representable matroids is
\[\{U_{2,5},U_{3,5},F_{7},F_{7}^{*},
F_{7}^{-},\nfd,\agde,(\agde)^{*},
\Delta_{3}(\agde),P_{8}\}.\]
\end{theorem}

Let $\mathcal{N}$ be the set featured in Theorem \ref{thm:nregexcluded}.

\begin{theorem}
\label{thm:nregsuper}
The only superfluous subsets of $\mathcal{N}$ are the
subsets of $\{F_{7},\agde,(\agde)^{*}\}$ and
$\{F_{7}^{*},\agde,(\agde)^{*}\}$.
The only $3$-connected matroids in
$\ex(\mathcal{N}-\{F_{7},\agde,(\agde)^{*}\})-\ex(\mathcal{N})$
are isomorphic to $F_{7}$, $\agde$,
$(\agde)^{*}$, $\AG(2,3)$, or $(\AG(2,3))^{*}$.
\end{theorem}

We note here that all undefined matroids appearing in the paper can be found in the appendix of Oxley \cite{oxley11}.
We assume that the reader is familiar with the terminology and notation from that source. We use the terms \emph{line} and \emph{plane} to
refer to rank-$2$ and rank-$3$ subsets of the ground set. 
By performing a $\Delta\text{-}Y$ exchange on $\agde$, we
obtain $\Delta_{3}(\agde)$, which is represented over
$\GF(3)$ by $[I_{4}\ A]$, where $A$ is the following matrix.
\begin{equation}
\label{agdematrix}
  A = \kbordermatrix{ & 5 & 6 & 7 & 8\\
        1 & 1 & 0 &-1 & 0\\
        2 & 1 & 0 & 1 & 1\\
        3 & 1 & 1 & 0 & 1\\
        4 & 0 & 1 & 1 & -1
       }.
\end{equation}

The paper is built up as follows. In Section \ref{sec:splitters} we use
Seymour's Splitter Theorem to prove that certain subsets are superfluous. 
To prove that a subset $\{M\}$ is not superfluous, we need to
generate an infinite number of $3$-connected matroids in
$\ex(\mathcal{E}-\{M\})-\ex(\mathcal{E})$.
We do so by the simple expedient of growing arbitrarily
long fans.
Section \ref{sec:growfans} proves the technical lemmas that allow us to do so. In Section \ref{sec:infinites} we introduce several matroids to which our method of growing fans will be applied, and in Section \ref{sec:final} we will round up the results. Note that the proofs in Sections \ref{sec:splitters} and \ref{sec:infinites} are finite case-checks that could be replaced by computer checks. However, at the moment of writing no sufficiently reliable software for this existed.

\section{Applying the splitter theorem}\label{sec:splitters}

The following result is very well-known \cite[Proposition~12.2.3]{oxley11}.

\begin{prop}
\label{prop:f7splitter2}
The matroid $F_{7}$ is a splitter for the class
$\ex(\{U_{2,4},F_{7}^{*}\})$.
\end{prop}

Our next result, which seems not to be in the literature, proves a generalization of Proposition \ref{prop:f7splitter2}.

\begin{theorem}\label{thm:F7splitter}
  The matroid $F_7$ is a splitter for the class
  $\ex(\{U_{2,5}, U_{3,5}, F_7^*\})$.
\end{theorem}

\begin{proof}
  By Seymour's Splitter Theorem we only have to check that $F_7$ has no $3$-connected single-element extensions and coextensions in $\ex(\{U_{2,5}, U_{3,5}, F_7^*\})$. If $M$ is a $3$-connected matroid such that $M\del e \cong F_7$, then either $e$ is on exactly one line of $F_7$, or $e$ is on no line of $F_7$. In either case $M\con e$ contains a $U_{2,5}$-minor.

Therefore we will assume that $M$ is a $3$-connected matroid such that $M\con e \cong F_7$ and $M$ belongs to $\ex(\{U_{2,5}, U_{3,5}, F_7^*\})$.
Let $\mathcal{M}$ be the class of matroids that are either
binary or ternary.
Now $\mathcal{M}$ is a minor-closed class, and its excluded
minors are characterized in~\cite{MOOW09}.
Certainly $M$ is not binary, since that would lead to a contradiction
to Proposition \ref{prop:f7splitter2}.
Moreover, $M$ is not ternary, as it contains an $F_{7}$-minor.
Therefore $M$ is not contained in $\mathcal{M}$.
Hence~\cite[Theorem~4.1]{SW96b} implies that $M$ contains a
$3$-connected excluded minor for $\mathcal{M}$.
There are only $4$ such excluded minors, and as
$M$ does not contain $U_{2,5}$ or $U_{3,5}$ as a minor,
$M$ must contain one of the matroids obtained from the
affine geometry $\AG(3,2)$ or
from $T_{12}$ by relaxing a circuit-hyperplane.
As $M$ contains only $8$ elements, $M$ must be isomorphic
to the unique
relaxation of $\AG(3,2)$.
But this matroid has an
$F_{7}^{*}$-minor (\cite[Page 646]{oxley11}).
This contradiction completes the proof.
\end{proof}

We can make short work of the case in which we do not exclude $P_8$.
Geelen et al. \cite[Theorem~1.5]{GOVW00} proved the following result:

\begin{theorem}
  If $M$ is a $3$-connected matroid in
  $\ex(\{U_{2,6}, U_{4,6}, P_6, F_7^-, \nfd\})$, and $M$ has a $P_8$-minor, then $M$ is a minor of $S(5,6,12)$. 
\end{theorem}

Since each of $U_{2,6}$, $U_{4,6}$, $P_6$ has a minor in
$\{U_{2,5}, U_{3,5}\}$, we immediately have

\begin{corollary}
\label{cor:P8splitter}
  If $M$ is a $3$-connected matroid in
  $\ex(\{U_{2,5}, U_{3,5}, F_7^-, \nfd\})$, and $M$ has a $P_8$-minor, then $M$ is a minor of $S(5,6,12)$.
\end{corollary}

Next, we determine what happens if we don't exclude $\agde$. 
Our starting point is the automorphism group of $\agde$.
Note that it is transitive on elements of the ground set (\cite[Page 653]{oxley11}).
For each element $p$ in $\agde$, there is a unique element $p'$ such that $p$ and $p'$ are not on a $3$-point line of $\agde$. Any automorphism will map $\{p,p'\}$ to another such pair, so specifying the image of $p$ also specifies the image of $p'$.
Consider automorphisms of the diagram in Figure~\ref{fig:AGdefig}
that pointwise fix~$1$ and~$8$. 
It is easy to confirm that the permutations below
(presented in cyclic notation),
\begin{equation}
\label{aut1}
(1)(2,4)(3,7)(5,6)(8)
\end{equation}
and
\begin{equation}
\label{aut2}
(1)(2,3,5)(4,6,7)(8)
\end{equation}
are two such automorphisms.
The next result follows easily from this discussion.

\begin{figure}[tbp]
  \centering
  \includegraphics{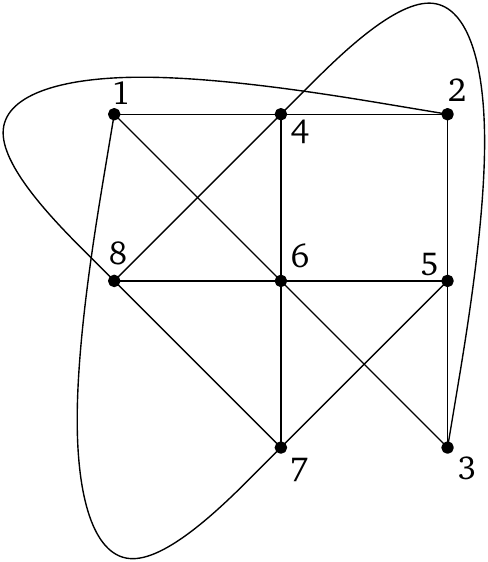}
  \caption{The matroid $\agde$.}
  \label{fig:AGdefig}
\end{figure}

\begin{lemma}\label{lem:transitivefixed}
  Let $p$ and $p'$ be points in $\agde$ such that there is no $3$-point line containing $p$ and $p'$. The subgroup of the automorphism group of $\agde$ that pointwise fixes $p$ and $p'$ is transitive on $E(\agde)\smin\{p,p'\}$.
\end{lemma}

\begin{lemma}\label{lem:transitivewhirlbasis}
  Let $B$ and $B'$ be bases of $\agde$ such that every pair $p,q \in B$, and every pair $k,l \in B'$ spans a $3$-point line. There is an automorphism of $\agde$ mapping $B$ to $B'$.
\end{lemma}

\begin{proof}
If $x$ is any element of $\agde$, then let $x'$ be the point
that is in no $3$-point line with $x$.
Let $B = \{p,q,r\}$.
The hypotheses of the lemma imply that
$|\{p,q,r,p',q',r'\}|=6$.
Let $e_{pq}$ be the unique point such that
$\{p,q,e_{pq}\}$ is a circuit.
Define $e_{pr}$ and $e_{qr}$ in the same way.
Then $|\{p,q,r,e_{pq},e_{pr},e_{qr}\}|=6$.
As $\agde$ contains only $8$ points, we
can relabel as necessary, and assume
$e_{qr}$ is in $\{p',q',r'\}$.
Since $e_{qr}$ is in a non-trivial line with $q$ and $r$,
it follows that $e_{qr}=p'$, so that
$\{p',q,r\}$ is a circuit.
Let $B' = \{k,l,m\}$.
By relabeling and using the same arguments, we can assume that
$\{k',l,m\}$ is a $3$-point line of $\agde$.

  Consider the automorphism that maps $k$ to $p$.
  It must map $k'$ to $p'$. By composing this
  automorphism with an automorphism that fixes $p$ and $p'$,
  and referring to Lemma \ref{lem:transitivefixed}, we can
   assume that $l$ is mapped to $q$. But an automorphism maps lines to lines, so then $m$ must be mapped to $r$, and the result follows.
\end{proof}

In the proof of the next lemma we will show several times that a matroid $M = M[I\ A]$ is isomorphic to one of $\Delta_{3}(\agde)$, $P_8$, $F_7^-$, or
$\nfd$. Unless the isomorphism is obvious (i.e.\ one merely needs to permute rows and columns), we will specify which isomorphism we use. For this we use the representation of $\Delta_{3}(\agde)$ with elements labeled as in \eqref{agdematrix}. Moreover, we will label the elements of $P_8$, $F_7^-$, $\nfd$ so that $P_8 = [I_4\ A_8]$, $F_7^- = [I_3\ A_7]$, and $\nfd = [-A_7^T\ I_4]$, where $A_{7}$ and $A_{8}$ are the following
matrices over $\GF(3)$.
\[
    A_8 = \kbordermatrix{ & 5 & 6 & 7 & 8\\
                        1 & 0 & 1 & 1 & -1\\
                        2 & 1 & 0 & 1 & 1\\
                        3 & 1 & 1 & 0 & 1\\
                        4 & -1 & 1 & 1 & 0} \qquad 
    A_7 = \kbordermatrix{ & 4 & 5 & 6 & 7\\
                        1 & 1 & 1 & 0 & 1\\
                        2 & 1 & 0 & 1 & 1\\
                        3 & 0 & 1 & 1 & 1 
                        }
\]

\begin{lemma}\label{lem:AG23-del-e}
  Let $M$ be a $3$-connected $\psru$-representable matroid such that $M\con f \cong \AG(2,3)\del e$ for some $f \in E(M)$. Then $M$ has $\Delta_{3}(\AG(2,3)\del e)$ as minor.
\end{lemma}

\begin{proof}
  Suppose that $M$ is a counterexample. Let $M' := M\del f$.
  \begin{claim}
    There exists a set $X\subseteq E(M)\smin f$ such that $|X| = 5$
    and $\rank(X)=3$.
      \end{claim}
  \begin{proof}
    Suppose $M'$ has no 5-point planes. First we show that $M'$ has no $3$-point lines. Observe that each line of $M'$ is a line of $\agde$, so $M'$ has no $4$-point lines. Suppose $\{x,y,z\}$ is a
    line of $M'$. If $x$ is on another 3-point line, then the union of those lines would be a 5-point plane, a contradiction. It follows that $M'\con x\del y$ is simple.
    Furthermore, $z$ is in no $3$-point line in $M'\con x\del y$,
    or else the union of this line with $\{x,y\}$ is a $5$-point plane
    in $M'$.
    Therefore $M'\con x\del y \con z$ is simple, has
    rank $2$, and contains $5$ points.
    Therefore $M'$ has a $U_{2,5}$-minor, which is impossible since it
    is $\psru$-representable.
        Hence $M'$ has no $3$-point lines.

Let $e$ be an arbitrary point in $E(M')$.
Then $M'\con e$ is a simple rank-$3$ matroid with $7$ points. Since $M'$ has no 5-point planes, $M'\con e$ has no 4-point lines. Hence $M'\con e$ cannot be the union of two lines, so it is $3$-connected.
Then $M'\con e$ is isomorphic to one of the matroids $F_{7}$,
$F_{7}^{-}$, $P_{7}$, or $O_{7}$
(see \cite[Page 292]{GGK}).
Since $M'\con e$ is $\psru$-representable, it is
not isomorphic to $F_{7}$ or $F_{7}^{-}$.
Furthermore, $O_{7}$ contains a $4$-point line, so $M'\con e$ must be isomorphic to $P_{7}$.
By the uniqueness of representation over $\GF(3)$, we can assume that
the following $\GF(3)$-matrix $A'$ is such that $M' = [I_4\ A']$.
\[
      A' := \kbordermatrix{
         & 4 & 5 & 6 & 7\\
      1  & 1      & 1     & 0      & -1\\
      2  & 1      & 0     & 1      & 1\\
      3  & 0      & 1     & 1      & 1\\
      e  & \alpha & \beta & \gamma & \delta}.
\]

As $M'$ has no $3$-point lines, all of $\alpha$, $\beta$, and
$\gamma$ are non-zero.
By scaling the row labeled $e$, we assume that $\alpha=1$. If $\gamma = \delta$ then $\{1,6,7\}$ is a triangle. It follows that $\gamma \neq \delta$.

If $\beta=1$, then $\gamma\ne 1$, or else
$M'\del 7\cong \nfd$.
Therefore $\gamma=-1$.
If $\delta=0$, then $A'$ represents $P_{8}$, which
is impossible as $M$ is $\GF(4)$-representable.
Therefore $\delta=1$. By the discussion above, $M'\con 1 \cong P_7$. But in $M'\con 1$, the sets $\{2,4,e\}$, $\{3,5,e\}$, and $\{6,7,e\}$ are triangles containing $e$, whereas $\{3,5,e\},\{4,5,6\}$, and $\{2,5,7\}$ are triangles containing $5$. This is a contradiction, since $P_7$ has only one element that is on three lines.
Therefore $\beta=-1$.
It follows that $\delta\ne 0$, or else $\{4,5,7\}$ is a triangle
of $M'$.

Assume that $\gamma=-1$, from which it follows that $\delta = 1$. Then we find that $M' \cong P_8$, with isomorphism
\[
1\to 1\quad
2\to 2\quad
3\to 5\quad
4\to 7\quad
5\to 8\quad
6\to 3\quad
7\to 6\quad
e\to 4.
\]
Therefore we must have $\gamma=1$, and hence $\delta = -1$. But then again $M' \cong P_8$, with isomorphism
\[
1\to 1\quad
2\to 5\quad
3\to 3\quad
4\to 8\quad
5\to 6\quad
6\to 2\quad
7\to 7\quad
e\to 4.
\]
From this final contradiction we conclude that the claim holds.
  \end{proof}

 Let $X$ be a set of 5 points of a plane of $M'$, and $Y := E(M') \smin X$. Note that $f\notin \closure_{M}(X)$, as $M\con f$ contains no rank-2 flat with 5 elements.
 
   Since $M\con f$ is isomorphic to $\agde$, we can distinguish three cases. Either $Y$ is a 3-point line of $M\con f$, or $Y$ is a basis of $M\con f$, and every pair of elements of $Y$ spans a $3$-point line in $M\con f$, or $Y$ is a basis of $M\con f$, and there is exactly one pair of elements in $Y$ that does not span a $3$-point line of $M\con f$. We can use Lemmas \ref{lem:transitivefixed} and \ref{lem:transitivewhirlbasis},
   and the fact that the automorphism group of $\agde$ is transitive on
   $3$-point lines (\cite[Page 653]{oxley11}), and thereby assume that either $Y = \{4,6,7\}$ or $Y = \{4,6,8\}$ or $Y = \{4,5,6\}$, where the elements of $\AG(2,3)\del e$ are labeled as in Figure \ref{fig:AGdefig}. 
     \paragraph{Case I} Suppose $Y = \{4,6,7\}$, so that $X=\{1,2,3,5,8\}$.
     Since $f$ is not a coloop and not in a series pair, there are two elements in $Y$ that are not spanned by $X$ in $M'$.
Let $\sigma$ be the automorphism in Equation \eqref{aut2}, so that
$Y$ is an orbit of $\sigma$.
There is some $i\in\{0,1,2\}$ such that $\sigma^{i}$ takes
the two elements in $Y-\closure_{M'}(X)$ to $\{4,6\}$.
Now $\sigma^{i}$ induces a relabeling of the elements of $M'$ that set-wise fixes $X$. After applying this relabeling, $M\con f$ is still equal to $\agde$, as labeled in Figure \ref{fig:AGdefig}. Moreover, $X$ is a $5$-point plane of $M'$ that does not span $4$ or $6$.
     By the uniqueness of representations over $\GF(3)$ we can assume that $M = M[I\ A]$ for some $\GF(3)$-matrix of the form
  \[
    A :=  \kbordermatrix{ & 4 & 5 & 6 & 7 & 8\\
                       f  & 1 & 0 & \alpha & \beta & 0\\
                       1  & 1 & 0 & 1 & 1 & 1\\
                       2  & 1 & 1 & 0 & -1 & 1\\
                       3  & 0 & 1 & 1 & -1 & -1}
  \]
  with $\alpha \neq 0$. If $\alpha = 1$ then $M\del \{5,7\} \cong (F_7^-)^*$, with isomorphism
  \[
  1 \to 5\quad
  2\to 7\quad
  3\to 6\quad
  4\to 4\quad
  6\to 2\quad
  8\to 3\quad
  f\to 1.
  \]
  Hence $\alpha = -1$. But now
  $M\del 7 \cong \Delta_{3}(\AG(2,3)\del e)$.
  This completes the analysis in Case I.
  
  From now on, we assume that $Y$ is not a triangle of $M/f$. We will also assume that if $X$ spans an element $y\in Y$, then there is no triangle $T$ of $M/f$ that contains $Y-y$. To justify this assumption, note that if $y\in \closure_{M'}(X)$, then $(Y-y)\cup f$ must be a triad of $M$, so that $\rank_{M}(X\cup y)=3$. Furthermore, $Y$ is not a triangle in $M/f$, so $T$ contains exactly one element of $X$. Therefore, if $T$ exists, we can replace $X$ with $(X-T)\cup y$, and replace $Y$ with $T$, and reduce to Case I.
  
  \paragraph{Case II} Suppose $Y = \{4,6,8\}$.
  Since any pair of elements from $\{4,6,8\}$ is in a triangle of $M/f$, we can assume that $X$ spans no element of $Y$, by the argument in the previous paragraph.
Hence we have $M = M[I\ A]$ for some $\GF(3)$-matrix of the form
  \[
    A :=  \kbordermatrix{ & 4 & 5 & 6 & 7 & 8\\
                       f  & 1 & 0 & \alpha & 0 & \beta\\
                       1  & 1 & 0 & 1 & 1 & 1\\
                       2  & 1 & 1 & 0 & -1 & 1\\
                       3  & 0 & 1 & 1 & -1 & -1},
  \]
  where $\alpha$ and $\beta$ are non-zero.

If $(\alpha,\beta)=(1,1)$, then $M\del 5 \cong \Delta_{3}(\agde)$, with isomorphism
\[
1\to 1\quad
2\to 2\quad
3\to 4\quad
4\to 3\quad
6\to 8\quad
7\to 7\quad
8\to 6\quad
f\to 5.
\]
If $(\alpha,\beta)=(1,-1)$, then $M\del 5 \cong P_8$, with isomorphism
\[
1\to 2\quad
2\to 3\quad
3\to 4\quad
4\to 6\quad
6\to 1\quad
7\to 5\quad
8\to 8\quad
f\to 7,
\]
contradicting $\GF(4)$-representability of $M$.

If $(\alpha,\beta)=(-1,1)$, then $M\con 1\del 5 \cong F_7^-$, with isomorphism
\[
  2\to 2\quad
  3\to 3\quad
  4\to 1\quad
  6\to 7\quad
  7\to 6\quad
  8\to 5\quad
  f\to 4.
\]

If $(\alpha,\beta)=(-1,-1)$, then $M\del 5 \cong \Delta_{3}(\agde)$, with isomorphism
\[
  1\to 2\quad
  2\to 7\quad
  3\to 5\quad
  4\to 4\quad
  6\to 3\quad
  7\to 6\quad
  8\to 8\quad
  f\to 1.  
\]
Thus $M$ has a $\Delta_{3}(\agde)$-minor.

    \paragraph{Case III} Suppose $Y = \{4,5,6\}$.
    Since $\{4,6,7\}$ and $\{5,6,8\}$ are triangles of $M/f$, we
    assume that neither $4$ nor $5$ is in the span of $X$, by the argument immediately preceding Case II.
    Hence $M = M[I\ A]$ for some $\GF(3)$-matrix of the form
  \[
    A :=  \kbordermatrix{ & 4 & 5 & 6 & 7 & 8\\
                       f  & 1 & \alpha & \beta & 0 & 0\\
                       1  & 1 & 0 & 1 & 1 & 1\\
                       2  & 1 & 1 & 0 & -1 & 1\\
                       3  & 0 & 1 & 1 & -1 & -1},
  \]
  where $\alpha\ne 0$.
If $\alpha = 1$ then $M\del \{6,8\} \cong (F_7^-)^*$, with isomorphism
  \[
    1\to 5\quad
    2\to 6\quad
    3\to 7\quad
    4\to 1\quad
    5\to 4\quad
    7\to 3\quad
    f\to 2.
  \]
  Therefore $\alpha=-1$. But now $M\del 6 \cong \Delta_{3}(\agde)$, with isomorphism
  \[
    1\to 8\quad
    2\to 3\quad
    3\to 2\quad
    4\to 7\quad
    5\to 1\quad
    7\to 4\quad
    8\to 6\quad
    f\to 5.
  \]
The result follows.
\end{proof}

We must now study coextensions of $\AG(2,3)$. Luckily our previous analysis can be used for this.

\begin{lemma}\label{lem:delelemAG23}
      Let $M$ be a $3$-connected $\psru$-representable matroid such that $M\con f \cong \AG(2,3)$ for some $f \in E(M)$. Then $M$ has an element $g \neq f$ such that $M\del g$ is $3$-connected.
\end{lemma}

\begin{proof}
    Let $M$ be as stated, and suppose the result is false, so for each element $g\neq f$, $M\del g$ is not 3-connected. Since $M\del g \con f$ is 3-connected, $g$ must be in a triad with $f$. Two distinct triads $T_1$ and $T_2$, both containing $f$, intersect only in $f$, or else $M/f\cong \AG(2,3)$ contains a triad. From this we find that $M\del f$ can be partitioned into series pairs. However, this matroid has an odd number of elements, a contradiction.
\end{proof}

\begin{corollary}
\label{cor:agde}
  Let $M$ be a $3$-connected $\psru$-representable matroid such that $M\con f \cong \AG(2,3)$ for some $f \in E(M)$. Then $M$ has $\Delta_{3}(\agde)$ as minor.
\end{corollary}

\begin{proof}
    Let $g$ be an element as in Lemma \ref{lem:delelemAG23}. Then $M\del g$ is a matroid satisfying all conditions of Lemma \ref{lem:AG23-del-e}, and the result follows.
\end{proof}

Now we combine the previous results and the Splitter Theorem
to prove the following theorem.

 \begin{theorem}\label{thm:agde}
  Let $M$ be a $3$-connected matroid in \[\ex(\{U_{2,5}, U_{3,5}, F_7, F_7^*, F_7^-, \nfd, \Delta_{3}(\agde), P_8\}).\] Then either $M$ is near-regular, or one of $M$ and $M^{*}$ is isomorphic to a member of
  $\{\agde, \AG(2,3)\}$.
\end{theorem}

\begin{proof}
By the excluded-minor characterization of $\psru$-representable
matroids (Theorem \ref{thm:sruexcluded}), it follows that
$M$ is $\psru$-representable.
We assume that $M$ is not $\nreg$-representable.
Then Theorem \ref{thm:nregexcluded} implies that
$M$ contains a minor isomorphic to $\agde$ or its dual.
By duality, we assume that $M$ has an $\agde$-minor.
If $M\cong \agde$, we are done, so we assume otherwise.
By Seymour's Splitter Theorem, $M$ contains a 
$3$-connected minor $M'$, such that $M'$ is a single-element
extension or coextension of $\agde$.
Lemma \ref{lem:AG23-del-e} implies that $M'$ is a
single-element extension of $\agde$.
Thus $M'$ is simple and $r(M')=3$.
Moreover $|E(M')|=9$, so \cite[Theorem 2.1]{OVW98}
implies that $M'\cong \AG(2,3)$.
If $M=M'$, we are done, so we assume that $M$ has a
$3$-connected minor $M''$, such that $M''$ is a
single-element extension or coextension of $\AG(2,3)$.
But $r(M'')>3$, or else we have contradicted
\cite[Theorem 2.1]{OVW98}.
Therefore $M''\con f\cong \AG(2,3)$, for some element $f$.
Corollary \ref{cor:agde} implies that $M''$ has a
$\Delta_{3}(\agde)$-minor, a contradiction.
\end{proof}

\section{Creating bigger fans}\label{sec:growfans}
In this section we prove two results that allow us to replace a fan by a bigger fan while keeping a certain minor $N$, without losing 3-connectivity, and without introducing an undesired minor $N'$ (subject to the conditions that $N'$ is 3-connected and has no 4-element fans). We will use Brylawski's generalized parallel connection \cite{Bry75} for this. We refer the reader to Oxley \cite[Section 11.4]{oxley11} for definitions and elementary properties, including the following:

\begin{lemma}\label{lem:genparconflats}
    Let $M$ and $N$ be matroids having a common restriction $T$, which is moreover a modular flat of $N$. Let $M' := P_T(N,M)$.
    \begin{enumerate}
        \item A subset $F\subseteq E(M')$ is a flat of $M'$ if and only if $F\cap E(N)$ is a flat of $N$ and $F\cap E(M)$ is a flat of $M$;
        \item $M'|E(N) = N$ and $M'|E(M) = M$;
        \item If $e\in E(N)\smin T$ then $M'\del e = P_T(N\del e, M)$;
        \item If $e\in E(N)\smin \closure_{N}(T)$ then $M'\con e = P_{T}(N\con e, M)$;
        \item If $e\in E(M)\smin T$ then $M'\del e = P_T(N, M\del e)$;
        \item If $e\in E(M)\smin \closure_M(T)$ then $M'\con e = P_T(N, M\con e)$.
    \end{enumerate}
\end{lemma}

Let $M$ be a matroid on the ground set $E$.
A subset of $E$ is \emph{fully-closed} if
it is closed in $M$ and $M^{*}$.
If $X\subseteq E$, then $\fullclosure(X)$ is the
intersection of all fully-closed sets that contain $X$.
We can obtain $\fullclosure(X)$ by applying
the closure operator to $X$, applying the coclosure operator
to the result, and so on, until we cease to gain any new elements.

\begin{lemma}\label{lem:simplecosimple2sep}
    Let $M$ be a simple, cosimple, connected matroid, and let $(A,B)$ be a $2$-separation of $M$. Then $(\fullclosure_M(A), B\smin \fullclosure_M(A))$ is a $2$-separation.
\end{lemma}

\begin{proof}
It is simple to verify that $\lambda_M(\fullclosure_M(A)) \leq \lambda_M(A)$.
If $(\fullclosure_M(A), B\smin \fullclosure_M(A))$ is not a
$2$-separation, then $|B\smin \fullclosure_M(A)|<2$.
This means that we can order the elements of $B$
as $(b_{1},\ldots, b_{k})$, so that 
$b_{i}$ is in $\closure_{M}(A\cup\{b_{1},\ldots, b_{i-1}\})$ or
$\coclosure_{M}(A\cup\{b_{1},\ldots, b_{i-1}\})$, for
all $i\in \{1,\ldots, k-1\}$.
Hence $\lambda_M(A\cup\{b_{1},\ldots, b_{k-2}\})\leq 1$, so
$1 \geq \lambda_M(\{b_{k-1},b_{k}\}) =
\rank_M(\{b_{k-1},b_{k}\}) + \corank_M(\{b_{k-1},b_{k}\}) - 2$.
Thus $\{b_{k-1},b_{k}\}$ is either dependent or codependent.
In either case we have a contradiction.
\end{proof}

\begin{definition}\label{def:fan}
    Let $M$ be a matroid, and $F = (x_1,x_2,\ldots,x_k)$ an ordered subset of $E(M)$, with $k \geq 3$. We say $F$ is a \emph{fan} of $M$ if, for all $i\in \{1, \ldots, k-2\}$, $T_i := \{x_i, x_{i+1}, x_{i+2}\}$ is either a triangle or a triad, and if $T_i$ is a triad, then $T_{i+1}$ is a triangle; if $T_i$ is a triangle then $T_{i+1}$ is a triad.
\end{definition}

Assume that $F=(x_{1},\ldots, x_{k})$ is a fan.
Then $F$ is a fan of $M^*$.
We say that $F$ is a \emph{maximal} fan if there is no
fan $(y_{1},\ldots, y_{l})$ such that $l > k$ and
$\{x_{1},\ldots, x_{k}\}\subseteq \{y_{1},\ldots, y_{l}\}$.
We say $x_i$ is a \emph{rim element} if $1<i<k$ and $x_{i}$ is contained in exactly one triangle that is contained in $\{x_{1},\ldots, x_{k}\}$, or if $i\in\{1,k\}$ and $x_{i}$ is contained in no such triangle.
We say $x_i$ is a \emph{spoke element} if it is not a rim element.

\begin{lemma}\label{lem:fan2sep}
    Let $M$ be a simple, cosimple, connected matroid, let $F = (x_1,\ldots, x_k)$ be a fan of $M$, and let $(A,B)$ be a $2$-separation of $M$. Then $M$ has a $2$-separation $(A',B')$ with $\{x_{1},\ldots, x_{k}\}\subseteq A'$.
\end{lemma}

\begin{proof}
    Let $M$, $F$, and $(A,B)$ be as stated. Assume, by dualizing $M$ if necessary, that $T := \{x_1,x_2,x_3\}$ is a triangle. Clearly one of $A\cap T$ and $B\cap T$ has size at least two; by relabeling assume $|A\cap T| \geq 2$. By Lemma \ref{lem:simplecosimple2sep}, we can replace $(A,B)$ with $(\fullclosure_M(A), B\smin\fullclosure_M(A))$. If $\{x_{1},\ldots, x_{k}\}\subseteq \fullclosure_{M}(A)$ then we are done, so assume that $i \in \{1, \ldots, k\}$ is the smallest index such that $x_i\notin \fullclosure_{M}(A)$.
    Certainly $T$ is contained in $\fullclosure_{M}(A)$, so $i \geq 4$. But then either $x_i \in \closure_M(\{x_{i-1},x_{i-2}\})$ or $x_i\in\coclosure_M(\{x_{i-1},x_{i-2}\})$, which leads to a contradiction.
\end{proof}

In what follows, the elements of the wheel $\wheel{n}$ and whirl $\whirl{n}$ are labeled $\{s_1,r_1,s_2,\ldots, s_n,r_n\}$ where, for all indices $i$ (interpreted modulo $n$), $\{s_i, r_i, s_{i+1}\}$ is a triangle and $\{r_i, s_{i+1}, r_{i+1}\}$ is a triad. Hence, $\{s_1, \ldots, s_n\}$ is the set of spokes and $\{r_1, \ldots, r_n\}$ is the set of rim elements.

\begin{theorem}\label{thm:growfan}
    Let $M$ be a $3$-connected matroid, and let $F = (x_1,\ldots,x_k)$ be a fan of $M$ with $T := \{x_1,x_2,x_3\}$ a triangle. Let $n \geq 3$ be an integer, and relabel the elements $s_1$, $r_n$, $s_n$ of $\wheel{n}$ by $x_1$, $x_2$, $x_3$, in that order. Let $M' := P_{T}(\wheel{n},M)$, and $M'' := M'\del x_2$. Then $M''$ has the following properties:
    \begin{enumerate}
        \item\label{eq:growfan1}$(x_1, r_1, s_2, r_2,\ldots,s_{n-1}, r_{n-1},x_3, \ldots, x_k)$ is a fan of $M''$;
                \item\label{eq:growfan2} $M$ is isomorphic to a minor of $M''$, with the isomorphism fixing all elements but $x_2$;
        \item\label{eq:growfan3} $M''$ is $3$-connected;
    \end{enumerate}
\end{theorem}

\begin{proof}
    Let $M$, $F$, $T$, $n$, $M'$, and $M''$ be as stated, and define $N := \wheel{n}$. It follows immediately from \cite[Corollary 6.9.10]{oxley11} that $T$ is a modular flat of $N$, so $M'=P_{T}(N,M)$ is defined. It follows from Lemma \ref{lem:genparconflats} that $(s_1, r_1, \ldots, s_{n-1},r_{n-1},s_n)$ is a fan of $M'$ and of $M''$, since the complement in $M'$ of each triad of $N$ is a hyperplane of $M'$, and each triangle of $N$ other than $T$ is a triangle of $M'$. If $k=3$, then \eqref{eq:growfan1} holds. Hence we assume that $k \geq 4$. We only need to show that $\{r_{n-1},s_n,x_4\}$ is a triad of $M''$. Consider $H := E(M')\smin \{r_{n-1},s_n,r_n,x_4\}$. Since $H\cap E(N)$ and $H\cap E(M)$ are hyperplanes of their respective matroids, $H$ is a flat of $M'$. Since $\closure_{M'}(H\cup s_n) = E(M')$, it follows that $\{r_{n-1},s_n,r_n,x_4\}$ is a cocircuit of $M'$. But then $\{r_{n-1},s_n,x_4\}$ is a cocircuit of $M''$, as desired.
    
Statement \eqref{eq:growfan2} is a straightforward consequence of Lemma \ref{lem:genparconflats} and the observation that $\wheel{n}$ has a minor in which $\{s_1,r_n,s_n\}$ is a triangle and some element is in parallel with $r_n$.
Statement \eqref{eq:growfan3} follows immediately from \cite[Corollary 2.8]{OW00}.
\end{proof}
We will denote the matroid $M''$, as described in the statement of Theorem \ref{thm:growfan}, by $\growfan{T}{n}{M}$.
Theorem \ref{thm:growfan} shows that we can make a fan arbitrarily long while keeping 3-connectivity. Our next task is to show that we can do so without introducing certain minors. The following lemma, whose elementary proof we omit, will be useful:

\begin{lemma}\label{lem:4eltfan:no}
    Let $N$ be a $3$-connected matroid without $4$-element fans. Let $M$ be a $3$-connected matroid having $N$ as minor, and let $F$ be a $4$-element fan of $M$. Then $|F\cap E(N)| \leq 3$.
\end{lemma}

Recall that if $T$ is a  coindependent triangle of the matroid $M$,
then $\Delta_{T}(M)$ is the matroid obtained from $M$ by a
$\Delta\text{-}Y$ exchange (see \cite[Section 11.5]{oxley11}).

\begin{theorem}\label{thm:nofanfan}
    Let $N$ be a $3$-connected matroid with no $4$-element fan. Let $M$ be a $3$-connected matroid with at least $5$ elements that does not have an $N$-minor. Let $F = (x_1, \ldots, x_k)$ be a fan of $M$, where $T:=\{x_{1},x_{2},x_{3}\}$ is a triangle, and let $n\geq 3$ be an integer.
If $\growfan{T}{n}{M}$ has an $N$-minor, then so does
$\Delta_{T}(M)$.
\end{theorem}

\begin{proof}
We will assume that $n\geq 3$ has been chosen so that it is as small as possible, subject to the condition that $\growfan{T}{n}{M}$ has an $N$-minor.
Let $N'$ be a minor of $\growfan{T}{n}{M}$ that is isomorphic to $N$.

First assume that $n\geq 4$.
Since $\{r_{1},s_{2},r_{2},s_{3}\}$ is a 4-element fan of $\growfan{T}{n}{M}$, it follows from Lemma \ref{lem:4eltfan:no} that this set is not contained in $E(N)$.
We claim that $\growfan{T}{n}{M}\con r_{1}\del s_{2}$ has an
$N$-minor.
Assume this is not the case.
If $\growfan{T}{n}{M}\con r_{1}$ has an $N$-minor, then,
as $\{s_{1},s_{2}\}$ is a parallel pair,
$\growfan{T}{n}{M}\con r_{1}\del s_{2}$ has an $N$-minor.
Therefore $\growfan{T}{n}{M}\con r_{1}$
does not have an $N$-minor.
Similarly, $\{r_{1},r_{2}\}$ is a series pair in
$\growfan{T}{n}{M}\del s_{2}$, so we assume that
$\growfan{T}{n}{M}\del s_{2}$ has no $N$-minor.
As $\{s_{2},s_{3}\}$ is a parallel pair in
$\growfan{T}{n}{M}\con r_{2}$, this means that
$\growfan{T}{n}{M}\con r_{2}$ has no $N$-minor.
Moreover, $\{r_{2},r_{3}\}$ is a series pair in
$\growfan{T}{n}{M}\del s_{3}$, so this matroid
does not have an $N$-minor.
As $\{s_{2},r_{2}\}$ is a series pair in
$\growfan{T}{n}{M}\del r_{1}$, and we concluded that
$\growfan{T}{n}{M}\con r_{2}$ has no $N$-minor, neither
does $\growfan{T}{n}{M}\del r_{1}$.
Since $\{r_{1},s_{1}\}$ is a parallel pair in
$\growfan{T}{n}{M}\con s_{2}$, and deleting
$r_{1}$ destroys all $N$-minors,
$\growfan{T}{n}{M}\con s_{2}$ has no $N$-minor.
Deleting $r_{2}$ creates the series pair $\{r_{1},s_{2}\}$,
and contracting $r_{1}$ destroys all $N$-minors, so
$\growfan{T}{n}{M}\del r_{2}$ does not have an $N$-minor.
Lastly, contracting $s_{3}$ creates the parallel pair
$\{s_{2},r_{2}\}$, so $\growfan{T}{n}{M}\con s_{3}$
does not have an $N$-minor, or else
$\growfan{T}{n}{M}\del s_{2}$ does.
From this discussion, we conclude that
$\{r_{1},s_{2},r_{2},s_{3}\}\subseteq E(N')$,
contradicting our earlier conclusion.
Therefore $\growfan{T}{n}{M}\con r_{1}\del s_{2}$
has an $N$-minor.

Since contracting $r_{1}$ and deleting $s_{2}$ from
$\wheel{n}$ produces a copy of $\wheel{n-1}$, it follows
easily from Lemma \ref{lem:genparconflats} that
$\growfan{T}{n}{M}\con r_{1}\del s_{2}$ is isomorphic to
$\growfan{T}{n-1}{M}$.
Thus our assumption on the minimality of $n$ is contradicted.
Now we must assume that $n=3$.

If $\{r_{1},s_{2},r_{2}\}\nsubseteq E(N')$, then as this set is
a triad in $\growfan{T}{n}{M}$, obtaining $N'$ involves contracting
one of $\{r_{1},s_{2},r_{2}\}$.
Contracting any of these elements in $\wheel{3}$ produces a
matroid consisting of the triangle $\{s_{1},r_{3},s_{3}\}$
with parallel points added to two distinct elements.
Now we can use Lemma \ref{lem:genparconflats} to show that
contracting an element in $\{r_{1},s_{2},r_{2}\}$
from $\growfan{T}{n}{M}$ produces a matroid that is isomorphic
to $M$ or $M\del x_{2}$, up to the addition of parallel elements.
Therefore $M$ has an $N$-minor, contrary to hypothesis.
It follows that $\{r_{1},s_{2},r_{2}\}\subseteq E(N')$.

Since $\{s_{1},s_{2},r_{2}\}$ is a triangle in
$\growfan{T}{n}{M}$, we deduce that $s_{1}\notin E(N')$, or else
$(s_{1},r_{1},s_{2},r_{2})$ is a $4$-element fan in
$N'$.
Since $\{r_{1},s_{2}\}$ is a parallel pair in
$\growfan{T}{n}{M}\con s_{1}$, and $N'$ contains no parallel pairs,
$N'$ is a minor of $\growfan{T}{n}{M}\del s_{1}$.
As $\{s_{2},r_{2},s_{3}\}$ is also a triangle in
$\growfan{T}{n}{M}$, we can use exactly the same arguments to show that
$N'$ is a minor of $\growfan{T}{n}{M}\del s_{3}$.
So $N'$ is a minor of
$P_{T}(\wheel{3},M)\del T$.
Since $|E(M)|\geq 5$, it is easy to prove that any triangle of $M$ is coindependent
(\cite[Lemma 8.7.5]{oxley11}).
Therefore $P_{T}(\wheel{3},M)\del T$ is
isomorphic to
$\Delta_{T}(M)$, and we are done.
\end{proof}

\section{Infinite families}\label{sec:infinites}
In this section we describe a collection of matroids to which
we can apply our operation of growing fans.
Recall that $\mathcal{O}$, $\mathcal{S}$, and $\mathcal{N}$, respectively, denote the sets of excluded minors for $\GF(4)$-representable,
$\psru$-representable, and
$\nreg$-representable matroids, as listed in
Theorems \ref{thm:gf4excluded},
\ref{thm:sruexcluded}, and \ref{thm:nregexcluded}.

Let $M_{8}$ be the rank-$3$ matroid shown in
Figure \ref{fig:M9}.
Then $M_{8}$ is represented over $\GF(3)$ by 
$[I_{3}\ A]$, where $A$ is the following matrix.
\[
\kbordermatrix{
 &4&5&6&7&8\\
1&1&1&0&1&0\\
2&1&0&1&1&1\\
3&0&1&1&1&-1\\
}
\]

\begin{figure}[htb]
    \centering
        \includegraphics[scale=1]{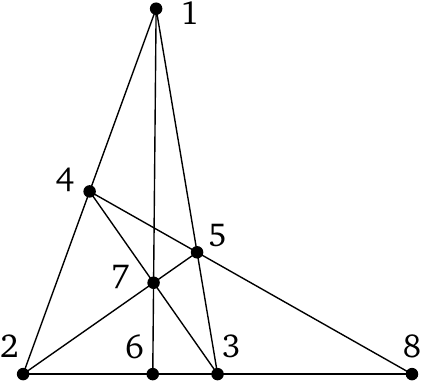}
    \caption{Geometric representation of $M_{8}$.}
    \label{fig:M9}
\end{figure}

\begin{lemma}\label{lem:M9}
Let $T$ be the triangle $\{3,6,8\}$ of $M_{8}$.
If $n\geq 3$ is an integer, then $\growfan{T}{n}{M_{8}}$
is $3$-connected, and contains an $F_{7}^{-}$-minor but no
minor in $(\mathcal{O}\cup\mathcal{S}\cup\mathcal{N})-\{F_{7}^{-}\}$.
\end{lemma}

\begin{proof}
Clearly $M_{8}$ is $3$-connected, and $(3,6,8)$ is a fan, so we
can apply Theorem \ref{thm:growfan}.
Thus $\growfan{T}{n}{M_{8}}$ is $3$-connected, by
statement \eqref{eq:growfan3}.
Since $M_{8}\del 8$ is isomorphic to $F_{7}^{-}$, it follows from
statement \eqref{eq:growfan2} that $\growfan{T}{n}{M_{8}}$
has an $F_{7}^{-}$-minor for any $n\geq 3$.

Now assume that $\growfan{T}{n}{M_{8}}$ has a minor in
$(\mathcal{O}\cup\mathcal{S}\cup\mathcal{N})-\{F_{7}^{-}\}$.
Therefore either $M_{8}$ or $\Delta_{T}(M_{8})$ has such a minor,
by Theorem \ref{thm:nofanfan}.
Since $M_{8}$ and $\Delta_{T}(M_{8})$ are both ternary
(\cite[Lemma 11.5.13]{oxley11}), neither has a minor isomorphic to
$U_{2,6}$, $U_{4,6}$, $P_6$, $P_8^{=}$,
$U_{2,5}$, $U_{3,5}$, $F_7$, or $F_7^*$.
As $\rank(M_{8})=3$, and $\rank(\Delta_{T}(M_{8}))=4$, neither
contains $(\agde)^*$.
Since $\rank(M_{8})=3$, and $\Delta_{T}(M_{8})$ contains the
triangle $\{2,5,7\}$ and has $8$ elements,
neither contains $P_{8}$.
As $\Delta_{T}(M_{8})$ has rank $4$, and $8$ elements, it
does not contain $\agde$.
As $M_{8}$ has $8$ elements and a $4$-point line, it does not
contain $\agde$.
Similarly, $M_{8}$ has rank $3$, so  it does not contain
$\Delta_{3}(\agde)$.
Also $\Delta_{T}(M_{8})$ has two triangles,
$\{2,5,7\}$ and $\{1,2,4\}$, so it does not contain
$\Delta_{3}(\agde)$ either.
The only matroid left to check is $\nfd$.
Obviously $M_{8}$ does not contain an $\nfd$-minor.
Assume that $\Delta_{T}(M_{8})$ does.
As $\nfd$ has no triangles, $\Delta_{T}(M_{8})\del 2$
must be isomorphic to $\nfd$.
Now $\{3,6,8\}$ is a triad of this matroid, and
performing a $Y\text{-}\Delta$ exchange on this triad
should produce a copy of $F_{7}^{-}$.
Instead it produces a copy of $M_{8}\del 2$, which
contains disjoint triangles, and is therefore not
isomorphic to $F_{7}^{-}$.
\end{proof}

Let $M_{9}$ be the matroid represented by $[I_{4}\ A]$
over $\GF(3)$, where $A$ is the following matrix.
\[
    \kbordermatrix{          & 5 & 6 & 7  & 8 & 9\\
                           1 & 1 & 0 & -1 & 0 & 1\\
                           2 & 1 & 0 & 1  & 1 & 1\\
                           3 & 1 & 1 & 0  & 1 & 0\\
                           4 & 0 & 1 & 1  & -1 &0
}
\]
Then $M_{9}$ is represented by the geometric diagram in
Figure \ref{fig:M11}.

\begin{figure}[htb]
    \centering
        \includegraphics[scale=1]{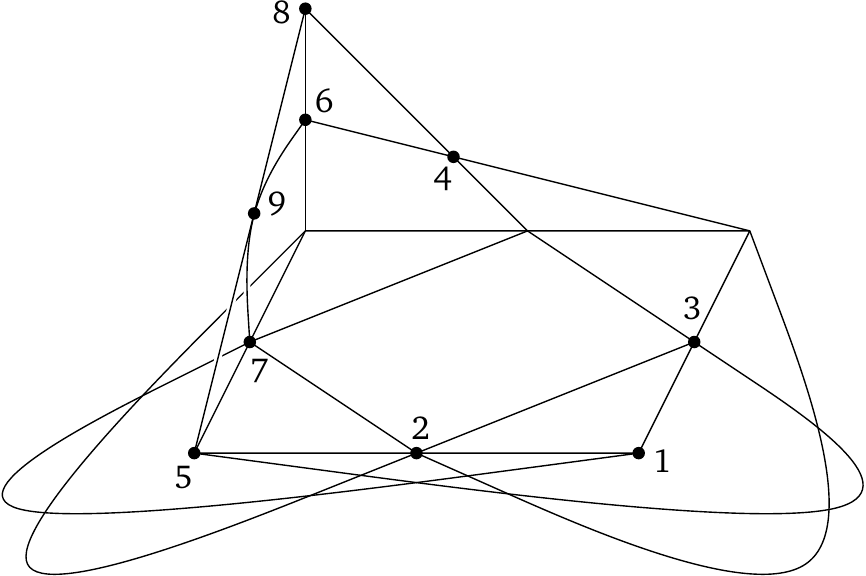}
    \caption{Geometric representation of $M_{9}$.}
    \label{fig:M11}
\end{figure}

\begin{lemma}\label{lem:M11}
Let $T$ be the triangle $\{3,5,9\}$ of $M_{9}$.
If $n\geq 3$ is an integer, then $\growfan{T}{n}{M_{9}}$
is $3$-connected, and contains an $\Delta_{3}(\agde)$-minor, but no
minor in $\mathcal{N}-\{\Delta_{3}(\agde)\}$.
\end{lemma}

\begin{proof}
It is clear that we can apply Theorem \ref{thm:growfan}.
Thus $\growfan{T}{n}{M_{9}}$ is $3$-connected.
Since $M_{9}\del 9$ is isomorphic to $\Delta_{3}(\agde)$,
it follows that $\growfan{T}{n}{M_{9}}$
has a $\Delta_{3}(\agde)$-minor for any $n\geq 3$.
If the lemma is false, then by Theorem \ref{thm:nofanfan},
either $M_{9}$ or $\Delta_{T}(M_{9})$ contains as a minor
a ternary member of $\mathcal{N}-\{\Delta_{3}(\agde)\}$, which is to say,
one of $F_{7}^{-}$, $\nfd$, $P_{8}$,
$\agde$, or $(\agde)^{*}$.

We start by noting that in $M_9\con 7$, the sets $\{3,5,8,9\}$ and $\{1,2,4,9\}$ are
4-point lines. Therefore any $7$-element restriction of $M_{9}\con 7$ has
either a $4$-point line or two disjoint triangles. It follows that $M_{9}\con 7$ has no minor in $\mathcal{N}$. As $\{2,3,4,6\}$ and $\{3,5,7,9\}$ are $4$-point lines of $M_{9}\con 8$, we can also see that $M_{9}\con 8$ has no minor in $\mathcal{N}$.

The triangles of $M_{9}$ are $\{1,2,9\}$, $\{3,5,9\}$, and $\{3,4,6\}$.
It follows easily that every $8$-element restriction of $M_{9}$ contains
at least one triangle, so $M_9$ does not have $P_8$ as minor. The rank of $M_9$ is too low to have $(\agde)^*$ as minor. If $M_9$ has $\agde$ as minor, then this minor must be obtained by a single contraction. Since $\agde$ is simple, we cannot contract an element from a 3-point line. This leaves only elements $7$ and $8$, and we have already decided that contracting either of these does not produce a minor in $\mathcal{N}$.

Suppose $M_{9}$ has a $\nfd$-minor. To obtain this minor we must delete two elements in such a way that no triangles remain. Since deleting $9$ gives us $\Delta_{3}(\agde)$ again, which has no $\nfd$-minor, we must delete $3$ and one of $\{1,2\}$. But $M_9\del \{1,3\}$ has disjoint triads $\{2,4,6\}$ and $\{5,7,9\}$, whereas $M_9\del \{2,3\}$ has disjoint triads $\{1,7,8\}$ and $\{4,5,9\}$. Hence neither is isomorphic to $\nfd$.

Therefore we assume that $M_{9}$ has an $F_{7}^{-}$-minor.
We must contract a single element from $M_{9}$, and then
delete a single element to obtain a copy of $F_{7}^{-}$.
If we contract either $3$ or $9$, then we produce
two disjoint parallel pairs, which cannot be rectified with a single deletion.
If we contract one of $1$, $2$, $4$, or $6$ then we create a single parallel pair, so up to isomorphism we must delete, respectively, $2$, $1$, $6$, or $4$ to obtain a copy of $F_{7}^{-}$.
But in these minors, the triangle $\{3,5,9\}$ is disjoint from,
respectively, the triangles
$\{6,7,8\}$, $\{4,6,8\}$, $\{1,2,7\}$, and $\{1,7,8\}$.
Therefore we do not contract $1$, $2$, $3$, $4$, $6$, $7$, $8$, or $9$.
If we contract $5$, then up to isomorphism we must delete $3$ to
obtain a copy of $F_{7}^{-}$, but in this minor
$\{1,4,8\}$ and $\{2,6,7\}$ are disjoint triangles.
Thus $M_{9}$ does not contain a minor in
$\mathcal{N}-\{\Delta_{3}(\agde)\}$.

Assume that $\Delta_{T}(M_{9})$ contains a minor $N'$ that
is isomorphic to a ternary member of $\mathcal{N}-\{\Delta_{3}(\agde)\}$.
If $T\nsubseteq E(N')$, then an element $x\in T$ is contracted
to obtain $N'$.
But $\Delta_{T}(M_{9})\con x\cong M_{9}\del x$, by
\cite[Lemma 2.13]{OSV00}, so $N'$ is isomorphic to a
minor of $M_{9}$.
Since this contradicts the conclusion of the previous paragraph, it follows
that $T$ is a triad of $N'$.
Therefore $N'$ is isomorphic to $\nfd$, or $(\agde)^{*}$.
It follows easily from \cite[Corollary 2.17]{OSV00}
and Seymour's Splitter Theorem, that
$\nabla_{T}(N')$ is a minor of
$\nabla_{T}(\Delta_{T}(M_{9}))=M_{9}$.
If $N'\cong \nfd$, then $\nabla_{T}(N)\cong F_{7}^{-}$,
and this leads to a contradiction.
Therefore $N'\cong (\agde)^{*}$.
The definition of $Y\text{-}\Delta$ exchange implies that
$\nabla_{T}(N')\cong (\Delta_{3}(\agde))^{*}$.
But $\Delta_{3}(\agde)$ is a self-dual matroid, so
$M_{9}$ has a minor isomorphic to $\Delta_{3}(\agde)$ that contains
$\{3,5,9\}$ in its ground set.
To obtain this minor, we must delete a single element,
but in each case the result has two triangles, namely $\{3,5,9\}$ and at least one of $\{1,2,9\}$ and $\{3,4,6\}$.
This is a contradiction as $\agde$ has only one triangle.
\end{proof}

For a third infinite class, consider the following matrix, $A$, over $\GF(8)$. Here $\alpha$ is an element that satisfies $\alpha^3 + \alpha + 1 = 0$.
Let $M_{7}$ be $[I_{3}\ A]$.
A geometric representation of $M_{7}$ can be found in Figure \ref{fig:M8}.

\[
    \kbordermatrix{ & 4 & 5 & 6 & 7\\
                           1 & 1 & 1 & 0 & 1\\
                           2 & 1 & 0 & 1 & \alpha\\
                           3 & 0 & 1 & \alpha & \alpha^{2}}
\]
\begin{figure}[htb]
    \centering
        \includegraphics[scale=1]{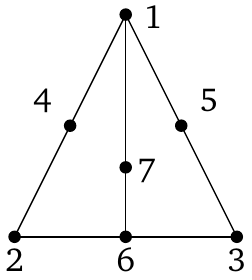}
    \caption{Geometric representation of $M_{7}$.}
    \label{fig:M8}
\end{figure}

\begin{lemma}\label{lem:M8}
Let $T$ be the triangle $\{1,2,4\}$ of $M_{7}$.
If $n\geq 3$ is an integer, then $\growfan{T}{n}{M_{7}}$
is $3$-connected, and contains a $P_{6}$-minor, but no
minor in $\mathcal{O}-\{P_{6}\}$.
\end{lemma}

\begin{proof}
Once again, the only way the lemma can fail is if
$M_{7}$ or $\Delta_{T}(M_{7})$ contains a minor in
$\mathcal{O}-\{P_{6}\}$.
(Note that $M_{7}\del 1\cong P_{6}$.)
As $M_{7}$ and $\Delta_{T}(M_{7})$ have only $7$
elements, any such minor must be isomorphic to
$U_{2,6}$, $U_{4,6}$, $F_{7}^{-}$, or $\nfd$.
But $M_{7}$ and $\Delta_{T}(M_{7})$ are $\GF(8)$-representable, and
therefore contain no $F_{7}^{-}$-minors or
$\nfd$-minors.
Obviously $M_{7}$ does not contain $U_{4,6}$.
It does not contain $U_{2,6}$, as every element is on
at least one triangle.
If $\Delta_{T}(M_{7})$ contains a $U_{2,6}$ minor, then
we must contract an element from $T$ to obtain
this minor, as $U_{2,6}$ has no triads.
But contracting an element from $T$ produces a
matroid isomorphic to a minor of $M_{7}$,
implying that $M_{7}$ has a $U_{2,6}$-minor.
Assume $\Delta_{T}(M_{7})$ has a $U_{4,6}$-minor.
By the previous argument, $T$ must be a triad of this
minor.
Then $M_{7}=\nabla_{T}(\Delta_{T}(M_{7}))$ has
a minor isomorphic to $\nabla_{T}(U_{4,6})\cong P_{6}$
that contains $T$.
But it is easy to see that the only way to obtain a
$P_{6}$-minor from $M_{7}$ is to delete the element
on three triangles, namely $1$.
As $1\in T$, this is a contradiction.
\end{proof}

\section{Proofs of the main results}\label{sec:final}

\begin{proof}[Proof of Theorem {\rm \ref{thm:regsuper}}]
If $M\in \ex(\{U_{2,4},F_{7}^{*}\})-\ex(\{U_{2,4},F_{7},F_{7}^{*}\})$
is $3$-connected, then $M$ has an $F_{7}$-minor, and
Proposition~\ref{prop:f7splitter2} implies that
$M$ is isomorphic to $F_{7}$.
Therefore $\{F_{7}\}$ is certainly superfluous.
As $\ex(\{U_{2,4},F_{7}\})-\ex(\{U_{2,4},F_{7},F_{7}^{*}\})$
consists of the duals of the matroids
in $\ex(\{U_{2,4},F_{7}^{*}\})-\ex(\{U_{2,4},F_{7},F_{7}^{*}\})$,
it follows that $\{F_{7}^{*}\}$ is also superfluous.
Since $\ex(\{F_{7},F_{7}^{*}\})-\ex(\{U_{2,4},F_{7},F_{7}^{*}\})$
contains all non-binary rank-$2$ uniform matroids, $\{U_{2,4}\}$
is contained in no superfluous subset.
Similarly,
$\ex(\{U_{2,4}\})-\ex(\{U_{2,4},F_{7},F_{7}^{*}\})$
contains all binary projective geometries.
Therefore $\{F_{7},F_{7}^{*}\}$ is contained in
no superfluous subset.
The result follows.
\end{proof}

\begin{proof}[Proof of Theorem {\rm \ref{thm:gf3super}}]
Theorem \ref{thm:F7splitter} implies that the only $3$-connected
matroid in
$\ex(\{U_{2,5},U_{3,5},F_{7}^{*}\})-\ex(\{U_{2,5},U_{3,5},F_{7},F_{7}^{*}\})$
is $F_{7}$ itself.
By duality, $F_{7}^{*}$ is the only $3$-connected matroid in
$\ex(\{U_{2,5},U_{3,5},F_{7}\})-\ex(\{U_{2,5},U_{3,5},F_{7},F_{7}^{*}\})$.
Thus $\{F_{7}\}$ and $\{F_{7}^{*}\}$ are superfluous subsets.
On the other hand,
$\ex(\{U_{3,5},F_{7},F_{7}^{*}\})-\ex(\{U_{2,5},U_{3,5},F_{7},F_{7}^{*}\})$
contains all the non-ternary rank-$2$ uniform matroids, so
$\{U_{2,5}\}$ is not contained in any superfluous subset.
Similarly,
$\ex(\{U_{2,5},F_{7},F_{7}^{*}\})-\ex(\{U_{2,5},U_{3,5},F_{7},F_{7}^{*}\})$
contains all the non-ternary corank-$2$ uniform matroids.
Finally,
$\ex(\{U_{2,5},U_{3,5}\})-\ex(\{U_{2,5},U_{3,5},F_{7},F_{7}^{*}\})$
contains all binary projective geometries, so
$\{F_{7},F_{7}^{*}\}$ is not superfluous.
\end{proof}

\begin{proof}[Proof of Theorem {\rm \ref{thm:GOVWcomplement}}]
Theorem \ref{thm:GOVW02} implies that if $M$ is a $3$-connected
matroid in $\ex(\mathcal{O}-\{P_{8},P_{8}^{=}\})-\ex(\mathcal{O})$,
then $M$ is isomorphic to $P_{8}^{=}$ or a minor of
$S(5,6,12)$.
Thus $\{P_{8},P_{8}^{=}\}$ is superfluous.
As
$\ex(\mathcal{O}-\{U_{2,6}\})-\ex(\mathcal{O})$
contains
all rank-$2$ uniform matroids with at least $6$ elements,
$\{U_{2,6}\}$, and by duality $\{U_{4,6}\}$, is not
contained in any superfluous subset.
By Lemma \ref{lem:M9}, the set
$\ex(\mathcal{O}-\{F_{7}^{-}\})-\ex(\mathcal{O})$
contains all matroids of the form $\growfan{T}{n}{M_{8}}$,
so $\{F_{7}^{-}\}$, and by duality $\{\nfd\}$,
is not contained in any superfluous subset.
Finally, Lemma \ref{lem:M8} shows that
$\ex(\mathcal{O}-\{P_{6}\})-\ex(\mathcal{O})$
contains an infinite number of $3$-connected matroids, so
$\{P_{6}\}$ is not contained in any superfluous subset.
\end{proof}

\begin{proof}[Proof of Theorem {\rm \ref{thm:srusuper}}]
Let $M$ be a $3$-connected matroid in
$\ex(\mathcal{S}-\{F_{7},P_{8}\})-\ex(\mathcal{S})$.
If $M$ has an $F_{7}$-minor, then
Theorem \ref{thm:F7splitter} implies that $M\cong F_{7}$.
Hence we assume that $M$ does not have an $F_{7}$-minor, so that
$M$ has a $P_{8}$-minor.
Corollary \ref{cor:P8splitter} says that $M$ is a
minor of $S(5,6,12)$.
Therefore $\{F_{7},P_{8}\}$ is superfluous.
Duality implies that the only $3$-connected matroids in
$\ex(\mathcal{S}-\{F_{7}^{*},P_{8}\})-\ex(\mathcal{S})$
are $F_{7}^{*}$, and minors of
$S(5,6,12)^{*}=S(5,6,12)$, so $\{F_{7}^{*},P_{8}\}$ is superfluous.
However,
$\ex(\mathcal{S}-\{U_{2,5}\})-\ex(\mathcal{S})$
contains infinitely many uniform matroids, and
$\ex(\mathcal{S}-\{F_{7}^{-}\})-\ex(\mathcal{S})$
contains all matroids of the form $\growfan{T}{n}{M_{8}}$.
Duality implies that none of
$\{U_{2,5}\}$, $\{U_{3,5}\}$, $\{F_{7}^{-}\}$,
$\{\nfd\}$ is contained in a superfluous subset.
Finally,
$\ex(\mathcal{S}-\{F_{7},F_{7}^{*}\})-\ex(\mathcal{S})$
contains all binary projective geometries, so
$\{F_{7},F_{7}^{*}\}$ is contained in no superfluous subset.
\end{proof}

\begin{proof}[Proof of Theorem {\rm \ref{thm:nregsuper}}]
Let $M$ be a $3$-connected matroid in
\[\ex(\mathcal{N}-\{F_{7},\agde,(\agde)^{*}\})-\ex(\mathcal{N}).\]
If $M$ contains an $F_{7}$-minor, then Theorem \ref{thm:F7splitter}
implies that $M\cong F_{7}$.
We assume that $M$ has no $F_{7}$-minor.
Then Theorem \ref{thm:agde} says that
$M$ is isomorphic to $\agde$, $\AG(2,3)$, or the
dual of one  of these matroids.
Therefore $\{F_{7},\agde,(\agde)^{*}\}$ is superfluous.
By duality, $\{F_{7}^{*},\agde,(\agde)^{*}\}$ is superfluous.
As
$\ex(\mathcal{N}-\{U_{2,5}\})-\ex(\mathcal{N})$
contains infinitely many uniform matroids, and
$\ex(\mathcal{N}-\{F_{7}^{-}\})-\ex(\mathcal{N})$
contains all matroids of the form
$\growfan{T}{n}{M_{8}}$, none of
$\{U_{2,5}\}$, $\{U_{3,5}\}$, $\{F_{7}^{-}\}$,
$\{\nfd\}$ is contained in a superfluous subset.
Moreover,
$\ex(\mathcal{N}-\{\Delta_{3}(\agde)\})-\ex(\mathcal{N})$
contains all matroids of the form
$\growfan{T}{n}{M_{9}}$, by Lemma \ref{lem:M11}.
Therefore $\{\Delta_{3}(\agde)\}$ is contained in no
superfluous subset.
Again, we observe that
$\ex(\mathcal{N}-\{F_{7},F_{7}^{*}\})-\ex(\mathcal{N})$
contains infinitely many binary matroids, so
the proof is complete.
\end{proof}

We conclude with the remark that, although our characterizations of $\ex(\mathcal{E})$ are strong when $\mathcal{E}$ contains all non-superfluous excluded minors in our class, we have made no attempt to characterize the infinite families. Clearly some of these families are highly structured. For instance, it is known that every rank-3 matroid with a $U_{2,5}$-minor also has a $U_{3,5}$-minor.

\subsection*{Acknowledgements.} Before writing our proofs we experimented to uncover the truth. These experiments were done using the Macek software by Hli\v{n}en\'{y} \cite{Hli04}, and occasionally we queried Mayhew and Royle's database of small matroids \cite{MR08}.

%\bibliography{matbib2009}

\begin{thebibliography}{10}

\bibitem{Bix79}
R.~E. Bixby,
\newblock On {R}eid's characterization of the ternary matroids,
\newblock J. Combin. Theory Ser. B 26 (1979) 174--204.

\bibitem{Bry75}
T.~Brylawski,
\newblock Modular constructions for combinatorial geometries,
\newblock Trans. Amer. Math. Soc. 203 (1975) 1--44.

\bibitem{GGK}
J.~F. Geelen, A.~M.~H. Gerards, A.~Kapoor,
\newblock The excluded minors for {$\mathrm{GF}(4)$}-representable matroids,
\newblock J. Combin. Theory Ser. B 79 (2000) 247--299.

\bibitem{GOVW00}
J.~F. Geelen, J.~G. Oxley, D.~L. Vertigan, G.~P. Whittle,
\newblock On the excluded minors for quaternary matroids,
\newblock J. Combin. Theory Ser. B 80 (2000) 57--68.

\bibitem{Hal43}
D.~W. Hall,
\newblock A note on primitive skew curves,
\newblock Bull. Amer. Math. Soc. 49 (1943) 935--936.

\bibitem{HMZ11}
R.~Hall, D.~Mayhew, S.~H.~M. van Zwam,
\newblock The excluded minors for near-regular matroids,
\newblock European J. Combin. 32 (2011) 802--830.

\bibitem{Hli04}
P.~Hlin{\v{e}}n{\'y},
\newblock Using a computer in matroid theory research,
\newblock Acta Univ. M. Belii Ser. Math. (2004) 27--44.

\bibitem{MOOW09}
D.~Mayhew, B.~Oporowski, J.~Oxley, G.~Whittle,
\newblock The excluded minors for the class of matroids that are binary or
  ternary,
\newblock European J. Combin. 32 (2011) 891--930.

\bibitem{MR08}
D.~Mayhew, G.~F. Royle,
\newblock Matroids with nine elements,
\newblock J. Combin. Theory Ser. B 98 (2008) 415--431.

\bibitem{oxley11}
J.~Oxley,
\newblock Matroid theory,
\newblock Oxford University Press, New York, second edition (2011).

\bibitem{OSV00}
J.~Oxley, C.~Semple, D.~Vertigan,
\newblock Generalized {$\Delta\text{-}Y$} exchange and {$k$}-regular matroids,
\newblock J. Combin. Theory Ser. B 79 (2000) 1--65.

\bibitem{OVW98}
J.~Oxley, D.~Vertigan, G.~Whittle,
\newblock On maximum-sized near-regular and $\sqrt[6]{1}$-matroids,
\newblock Graphs and Combinatorics 14 (1998) 163--179.

\bibitem{OW00}
J.~Oxley, H.~Wu,
\newblock On the structure of 3-connected matroids and graphs,
\newblock European J. Combin. 21 (2000) 667--688.

\bibitem{PZ08lift}
R.~A. Pendavingh, S.~H.~M. van Zwam,
\newblock Lifts of matroid representations over partial fields,
\newblock J. Combin. Theory Ser. B 100 (2010) 36--67.

\bibitem{SW96}
C.~Semple, G.~Whittle,
\newblock Partial fields and matroid representation,
\newblock Adv. in Appl. Math. 17 (1996) 184--208.

\bibitem{SW96b}
C.~Semple, G.~Whittle,
\newblock On representable matroids having neither {$U\sb {2,5}$}- nor {$U\sb
  {3,5}$}-minors,
\newblock in Matroid theory (Seattle, WA, 1995), volume 197 of
  Contemp. Math., pp. 377--386. Amer. Math. Soc., Providence, RI (1996).

\bibitem{Sey79}
P.~D. Seymour,
\newblock Matroid representation over {$\mathrm{GF}(3)$},
\newblock J. Combin. Theory Ser. B 26 (1979) 159--173.

\bibitem{Sey80}
P.~D. Seymour,
\newblock Decomposition of regular matroids,
\newblock J. Combin. Theory Ser. B 28 (1980) 305--359.

\bibitem{Tuthom}
W.~T. Tutte,
\newblock A homotopy theorem for matroids. {I}, {II},
\newblock Trans. Amer. Math. Soc. 88 (1958) 144--174.

\bibitem{Whi95}
G.~Whittle,
\newblock A characterisation of the matroids representable over
  {$\mathrm{GF}(3)$} and the rationals,
\newblock J. Combin. Theory Ser. B 65 (1995) 222--261.

\bibitem{Whi97}
G.~Whittle,
\newblock On matroids representable over {$\mathrm{GF}(3)$} and other fields,
\newblock Trans. Amer. Math. Soc. 349 (1997) 579--603.

\end{thebibliography}
%\bibliographystyle{MRStyle}

\end{document}